\documentclass{amsart}
\usepackage{amsmath,amssymb}
\usepackage{graphicx,subfigure}
\usepackage{color}

\newtheorem{theorem}{Theorem}[section]
\newtheorem{proposition}[theorem]{Proposition}

\newtheorem{lemma}[theorem]{Lemma}

\theoremstyle{definition}

\newtheorem*{definition*}{Definition}

\theoremstyle{remark}
\newtheorem{remark}[theorem]{Remark}

\numberwithin{equation}{section}

\newcommand{\de}{\delta}
\newcommand{\ep}{\varepsilon}

\newcommand{\ka}{\kappa}
\newcommand{\la}{\lambda}
\newcommand{\om}{\omega}

\newcommand{\De}{\Delta}

\newcommand{\La}{\Lambda}

\newcommand{\Om}{\Omega}

\newcommand{\tF}{\widetilde{F}}

\newcommand{\ha}{\widehat{a}}
\newcommand{\hb}{\widehat{b}}
\newcommand{\hF}{\widehat{F}}
\newcommand{\hL}{\widehat{L}}

\def\RR{\mathbb{R}}
\def\BB{\mathbb{B}}

\def\PP{\mathbb{P}_{p,\bla,\ep B}}

\def\Lae{\La_{p,\bla,\ep B}}
\def\Laz{\La_0}
\def\Laez{\La_{p,\bla,0}}

\renewcommand\SS{\mathbb{S}}
\newcommand{\cE}{{\mathcal E}}
\newcommand{\cF}{{\mathcal F}}
\newcommand{\cG}{{\mathcal G}}
\newcommand{\cH}{{\mathcal H}}

\newcommand{\cK}{{\mathcal K}}

\newcommand{\cP}{{\mathcal P}}

\newcommand{\cX}{{\mathcal X}}

\newcommand{\pd}{\partial}
\newcommand\minus\backslash

\newcommand\lan\langle
\newcommand\ran\rangle

\renewcommand\leq\leqslant
\renewcommand\geq\geqslant

\newcommand\Dir{_{\mathrm{Dir}}}

\newlength{\intwidth}

\newcommand\Be{B_{\ep,p,\bla}}

\newcommand\cGp{\cG_{p,\bla}}
\newcommand\cFp{\cF_{p,\bla}}
\newcommand\ueB{u_{p,\ep,\bla,B}}
\newcommand\ueBe{u_{p,\ep,\bla,\Be}}

\newcommand\Hp{H_{p,\bla}}
\newcommand\Tp{T_{p,\bla}}
\newcommand\Ome{\Om_{p,\ep B}}
\newcommand\bc{{\bar c}}
\newcommand\bla{{\bar\la}}
\newcommand\tb{\tilde b}
\newcommand\tf{\tilde f}

\newcommand\tL{\widetilde L}
\newcommand\chip{\chi_{p,\ep B}}
\newcommand\cHp{\cH_{p,\bla,B}}
\newcommand\Wp{W_{p,\bla}}
\newcommand\Vp{V_{p,\bla}}
\newcommand\Ye{Y_{\ep,\bla}}
\newcommand\pe{p_{\ep,\bla}}
\newcommand\phip{\phi_{p,\bla}}

\begin{document}

\title[Overdetermined boundary problems with
nonconstant data]{Overdetermined boundary problems with
nonconstant Dirichlet and Neumann data}

   % author one information
\author{Miguel Dom\'{\i}nguez-V\'{a}zquez}
\address{Departamento de Matem\'aticas, Universidade de Santiago de
  Compostela, Spain.}

\email{miguel.dominguez@usc.es}
   % author two information
\author{Alberto Enciso}
 \address{Instituto de Ciencias Matem\'aticas, Consejo Superior de
 	Investigaciones Cient\'ificas, Madrid, Spain.}
\email{aenciso@icmat.es}
% \email{aenciso@icmat.es}
   % author two information
\author{Daniel Peralta-Salas}
\address{Instituto de Ciencias Matem\'aticas, Consejo Superior de
 	Investigaciones Cient\'ificas, Madrid, Spain.}
\email{dperalta@icmat.es}

%%    General info
%\subjclass[2010]{35B38, 58J05, 58K45}
%\date{\today}
%
%\keywords{ }
%
\begin{abstract}
  In this paper we consider the overdetermined boundary problem for a
  general second order semilinear elliptic equation on bounded domains
  of~$\RR^n$, where one prescribes both the Dirichlet and Neumann data
  of the solution. We are interested in the case where the data are
  not necessarily constant and where the coefficients of the equation
  can depend on the position, so that the overdetermined problem does
  not generally admit a radial solution. Our main result is that, nevertheless,
  under minor technical hypotheses nontrivial solutions to the
  overdetermined boundary problem always exist.
\end{abstract}
\maketitle

\section{Introduction}

The study of overdetermined boundary problems, that is, problems where
one prescribes both Dirichlet and Neumann data, has grown into a major
field of research in the theory of elliptic PDEs since its appearance
in Lord Rayleigh's classic treatise~\cite{Rayleigh}. An outburst of
activity started with a groundbreaking paper of Serrin~\cite{Serrin},
where he combined an adaptation of Alexandrov's moving planes method with a
subtle refinement of the maximum principle to prove a symmetry result
for an overdetermined problem. More precisely, Serrin proved that, under mild technical
hypotheses, nontrivial
solutions to elliptic equations of the form
\[
\De u + F(u)=0
\]
inside a bounded domain~$\Om\subset\RR^n$ satisfying the boundary
conditions
\begin{equation}\label{standard}
u=0\quad\text{and}\quad \pd_\nu u=-c\qquad \text{on } \pd\Om\,,
\end{equation}
where $c$ is an unspecified constant that can be picked freely, only exist if~$\Om$ is a ball, in which case~$u$ is radial. The result remains true if~$F$ also depends on the
norm of the gradient of~$u$ and if we replace the Laplacian by other position-independent operators of variational form~\cite{Salani}.

The influence of Serrin's result is such that the very considerable
body of literature devoted to overdetermined boundary problems is
mostly limited to proofs that solutions need to be radial in cases
that can be handled using the method of moving planes. Without
attempting to be comprehensive, some remarkable
results about overdetermined boundary value problems include
alternative approaches to radial symmetry results using
$P$-functions (see e.g.\ the review~\cite{Kawohl}) or Pohozaev-type integral identities~\cite{Brandolini,MP1,MP2}, and extensions of the moving
plane method to the hyperbolic space and the hemisphere~\cite{KP98},
to degenerate elliptic equations such as the $p$-Laplace
equation~\cite{DPR99}, and to exterior~\cite{AB}, unbounded~\cite{Farina} or non-smooth
domains~\cite{Pr98}. Another
direction of research that has attracted considerable recent attention
is the study of connections with the theory of constant mean curvature
surfaces and the construction of nontrivial solutions to Serrin-type
problems in exterior domains~\cite{Traizet,DP,Ros1}. Nontrivial
solutions for partially overdetermined problems or with degenerate ellipticity are also known to exist~\cite{Fragala1,Fragala2}.

In two surprising papers, Pacard and Sicbaldi~\cite{PS:AIF}
and Delay and Sicbaldi~\cite{DS:DCDS} proved the existence of extremal
domains with small volume for the first eigenvalue of the
Laplacian in any compact Riemannian manifold, which guarantees the
existence of solutions to the overdetermined problem
for the linear elliptic equation
\[
  \De_g u +\la u=0
\]
in a domain with both zero Dirichlet data and constant
Neumann data. Here $\De_g$ is the Laplacian operator associated with a
Riemannian metric~$g$ on a compact
manifold and the constant~$\la$ (which one eventually chooses as the
first Dirichlet eigenvalue of the domain~$\Om$) is not specified a priori. Very recently we managed to show the existence of
nontrivial solutions, with the same overdetermined Dirichlet and Neumann conditions, for
fairly general semilinear elliptic equations of second order with possibly nonconstant
coefficients~\cite{overdet}.

In all these results, the fact that one is imposing precisely the standard overdetermined
boundary conditions~\eqref{standard} plays a crucial role. Roughly
speaking, this is because one can relate the existence of overdetermined
solutions with the critical points of certain functional via a
variational argument. Therefore,
the gist of the argument in these papers is that the overdetermined condition with
constant data is connected with the local extrema for a natural energy
functional, restricted to a specific class of functions labeled by
points in the physical space. This
ultimately permits to derive the existence of solutions from the fact
that a continuous function attains its maximum on a compact
manifold. However, this strategy is successful only for constant boundary data, and we are not
aware of an analog of this connection for general boundary data.

In the recent paper~\cite{Euler}, we have constructed new families of
compactly supported stationary solutions to the 3D Euler equation by
proving that there are solutions to an associated overdetermined problem in two
dimensions where one prescribes (modulo constants that
can be picked freely) zero Dirichlet data and nonconstant Neumann
data. The proof uses crucially that the space is
two-dimensional, which ensures that the kernel and cokernel of a
certain operator are one-dimensional, and does not work in higher
dimensions.

Our objective in this paper is to prove the existence of solutions to
overdetermined problems where one prescribes general Dirichlet and
Neumann data (just as before, up to unspecified constants). For
concreteness, we consider the model semilinear equation
\begin{equation}\label{eq}
  L u +\la F(x,u)=0
\end{equation}
in a bounded domain $\Om\subset \RR^n$, with Dirichlet and Neumann
boundary conditions
\begin{equation}\label{BCs}
  u= f_0(x)\,,\quad \nu\cdot A(x)\nabla u= -cf_1(x)\qquad \text{on }\pd\Om\,.
\end{equation}
Here $f_0,f_1$ are functions on~$\RR^n$, $F$ is a function
on~$\RR^n\times\RR$, $\la,c$ are
unspecified positive constants, $\nu$ is the outwards unit normal on $\partial\Omega$ and $L$ is the second-order
operator
\[
Lu:= a_{ij}(x) \,\pd_{ij} u + b_i(x)\, \pd_i u\,,
\]
where $A(x)=(a_{ij}(x))$ is a (symmetric)
matrix-valued function on~$\RR^n$ satisfying the (possibly
non-uniform) ellipticity
condition
\[
\min_{|\xi|=1}\xi\cdot A(x)\xi>0 \quad \text{for all }x\in\RR^n\,.
\]

\begin{theorem}\label{T.1}
Given any non-integer $s>2$, let us take any functions $F,f_0,f_1,b$ of
class~$C^{s}$ and $A$ of class $C^{s+2}$. Assume that the functions~$F(\cdot,f_0(\cdot))$ and~$f_1$ are positive and
that the function~$ f_0$ has a nondegenerate
critical point. Then there is a family of
domains~$\Om_{\ep,\bla}$ for which the overdetermined
problem~\eqref{eq}-\eqref{BCs} admits a solution.

More precisely, let $p\in\RR^n$ be a nondegenerate critical point
of~$f_0$. Then, for any small enough~$\ep\neq0$ and $\bla>0$, the following
statements hold:
\begin{enumerate}
\item The domain~$\Om_{\ep,\bla}$ is a small deformation of the ball of
  radius~$\ep$ centered at~$p$, characterized by an equation of the
  form $|x-p|^2<\ep^2+O(\ep^3)$.

  \item The dependence of $\la$ and~$c$ on the parameter~$\ep$ is of
    the form
    \[
      \la =
      \ep^{-2}\bla \,,\qquad c= \ep^{-1}\bc\,,
    \]
    where $\bc=\bc(\ep,\bla)$ is a positive constant of order~1.
\end{enumerate}
\end{theorem}

\begin{remark}\label{R.Gka}
In the case of the torsion problem, i.e., $\Delta u+\la=0$ (i.e., $F(x,u)=1$ in the previous notation),
the condition that $f_0$ has a critical point can be relaxed:
it is enough that the function $G_\ka:= f_0+\ka \log f_1$ has at least one nondegenerate
critical point for some constant~$\ka>0$. The statement then applies
if~$p$ is a nondegenerate critical point of~$G_\ka$ and taking $\bla:=
n\ka >0$ (not necessarily small).

Also, it is easy to obtain different variations on our main theorem
following the same method of proof. In fact, one obtains new results even for the linear
equation $\Delta u+b(x)\cdot \nabla u+\la f(x)=0$ with standard
overdetermined boundary data $f_0:=0$, $f_1:=1$; specifically, if $p$ is a nondegenerate
zero of the vector field~$n\nabla f-fb$, then the statement still holds taking any $\bla>0$. This does not
follow from~\cite{overdet}. However, we shall not
pursue these generalizations here.
\end{remark}

Compared with~\cite{overdet}, a major difference is that the theorem
does not only ensure the existence of domains where the overdetermined
problem under consideration admits nontrivial solution, but also
specifies the points around which those domains are located. This
immediately permits to translate this existence result to problems
that are only defined in a subset of~$\RR^n$ or on a differentiable manifold.

The paper is organized as follows. We will start by setting up the
problem in Section~\ref{S.setting}. For clarity of exposition, in Sections~\ref{S.setting} to~\ref{S.T2} we have
chosen to assume that the matrix $A(x)$ is the identity and carry out
the proof in this context. An essential ingredient of the proof is the
computation of asymptotic expansions for the solution to the Dirichlet
problem in small perturbations of a ball of radius~$\ep\ll1$, when the
constants~$\la$ and~$c$ scale with the radius as in
Theorem~\ref{T.1}. This computation is carried out in
Section~\ref{S.Asympt}. These asymptotic estimates are put to use in
Section~\ref{S.T2}, where we prove Theorem~\ref{T.1} in the particular
case when $A(x)=I$. To obtain the general result, in Section~\ref{S.A}
we show that the case of a general matrix-valued function~$A(x)$
reduces to the study of the easiest case $A(x)=I$ subject to an
inessential perturbation of order~$\ep^2$. Making this precise,
however, involves using a heavier notation and geodesic-type normal
coordinates adapted to the matrix~$A(x)$ that might unnecessarily obscure the simple ideas
the proof is based on. As a side remark, let us point out that the
reason we ask for more regularity of the matrix~$A$ (which is of class
$C^{s+2}$ in contrast with the $C^s$ regularity of the other
functions) is precisely due to our use of geodesic coordinates.

\section{Setting up the problem}
\label{S.setting}

For clarity of exposition, until Section~\ref{S.A} we will assume that $A(x)=I$.
This assumption will enable us to
obtain more compact expressions for the various quantities that appear
in the problem and it will make it easier to point out the salient
features of the proof.

Let us fix a point $p\in\RR^n$ and introduce rescaled
coordinates $z\in\RR^n$ centered at~$p$ as
\[
z:=\frac{x-p}{\ep}\,,
\]
where $\ep$ is a suitably small nonzero constant. We now consider
spherical coordinates $(r,\om)\in\RR^+\times\SS $ for~$z$, defined as
\[
r:= |z|= \left|\frac{x-p}{\ep}\right|\,,\qquad \om:= \frac z{|z|}=\frac{x-p}{|x-p|}\,.
\]
Here and in what follows,
\[
\SS:=\{\om\in\RR^n:|\om|=1\}
\]
denotes the unit sphere of dimension $n-1$.
For simplicity of notation, we will notationally omit the dependence
on the point~$p$. Also, with some abuse of notation, we will denote the expression of the
function~$u(x)$ in these coordinates simply by $u(r,\om)$.

Let us now consider a $C^{s+1}$~function $B:\SS \to\RR$ and, for
suitably small~$\ep$, let us describe the domain in terms of the above
coordinates as
\begin{equation}\label{Ome}
\Ome:=\{ r<1+\ep B(\om)\}\,.
\end{equation}
We now consider Equation~\eqref{eq} in the domain~$\Ome$ and
choose the constants $\la,c$ as
\[
\la=: \ep^{-2}\bla\,,\qquad c=:\ep^{-1} \bc\,,
\]
where we think of~$\ep$ as a small constant and of $\bla,\bc$ as
positive constants of order~1. Equation~\eqref{eq} can then be rewritten in the rescaled coordinates as
\begin{equation}\label{eqz}
\tL u + \bla \tF(z,u) =0\,,
\end{equation}
where
\[
  \tF(z,u):= F(p+\ep z,u)
\]
and $\tL$ is the differential operator
\[
  \tL u= \De u +\ep \tb(z)\cdot\nabla u\,,
\]
with $\tb_i(z):= b_i(p+\ep
z)$. We also denote the functions $f_0$ and $f_1$ in these coordinates as
\[
\widetilde f_0(z):=f_0(p+\ep z)\,, \qquad \widetilde f_1(z):=f_1(p+\ep z)\,.
\]
Here and in what follows, $\De$ and~$\nabla$ denote the Laplacian and
gradient operators in the rescaled coordinates~$z$.

The Dirichlet boundary condition on~$\pd\Ome$ can be
simply written in rescaled hyperspherical coordinates as
\begin{equation}\label{DBC}
u(1+\ep B(\om),\om)=\widetilde f_0(1+\ep B(\om),\om)=:\widehat f_0(\ep,\om)\,.
\end{equation}
We notice that $\widehat f_0(0,\om)=f_0(p)$. Analogously, the Neumann boundary condition reads as
\[
\pd_\nu u(1+\ep B(\om),\om)=-\bc \widetilde f_1(1+\ep B(\om),\om)\,,
\]
where $\nu$ is the outwards normal unit vector on $\pd\Ome$.

We denote by $C^s\Dir(\BB)$ the space of $C^s$~functions on the unit $n$-dimensional
ball $\BB:=\{|z|<1\}$ with zero trace to the boundary. Also,
$\cK\subset C^\infty(\SS )$ denotes the restriction to the unit sphere
of the space of linear functions on~$\RR^n$,
\[
\cK:= \{ V\cdot z: |z|=1,\; V\in\RR^n\}\,.
\]
Equivalently, $\cK$ is the eigenspace of the Laplacian
$\De_{\SS }$ of the unit sphere corresponding to the second
eigenvalue, $n-1$. Also, in what follows we will denote the partial
derivatives of~$F$ (or $\tF$) as
\[
F'(x,u):=\pd_u F(x,u)\,,\quad \nabla F(x,u):= \nabla_x F(x,u)\,,\quad \pd_jF(x,u):= \pd_{x_j} F(x,u)\,.
\]
The following lemma is a reformulation of~\cite[Theorem 2.3 and
Proposition 2.4]{overdet}.

\begin{lemma}\label{L.local}
  For each $p\in\RR^n$, there is some $\bla_p>0$ such that the
  following statements hold for all $\bla\in (0,\bla_p)$:
  \begin{enumerate}
    \item There is a unique function $\phip(r)$ of class $C^{s+2}$
satisfying the ODE
\[
\phip''(r)+\frac{n-1}r\phip'(r)+\bla F(p,f_0(p)+\phip(r))=0
\]
and the boundary condition $\phip(1)=0$ which is regular at
$r=0$. The function $\phip$ is well defined for $r\in[0,1+\de_p]$,
with $\de_p>0$. Furthermore, $\phip(r)>0$ for $r<1$ and $\phip'(1)<0$.

\item The operator
\[
\Tp v:= \De v +\bla \,F'(p, f_0(p)+\phip(|z|)) v
\]
defines an invertible map $\Tp: C^{s+1}\Dir(\BB)\to C^{s-1}(\BB)$.

\item Consider the map $\Hp$  defined for each
  function $\psi$ on the boundary of the ball as
  \[
    \Hp  \psi:= -\phip'(1)\, \pd_\nu v_\psi + \phip''(1) v_\psi
  \]
  where $v_\psi$ is the only solution to the problem $\Tp v_\psi=0$
  on~$\BB$, $v_\psi|_{\pd\BB}= \psi$. Then $\Hp$ maps
  $C^{s+1}(\SS ) \to C^s(\SS )$, its kernel is~$\cK$, and its range is
  the set~$C^s(\BB)\cap\cK^\perp$ of $C^s$~functions orthogonal
  to~$\cK$. Furthermore,
  \begin{equation}\label{estHp}
    \|\psi\|_{C^{s+1}}\leq C_{p,\bla }\|\Hp
    \psi\|_{C^{s}}
  \end{equation}
  for all~$\psi\in C^{s+1}\cap\cK^\perp$.

\item The function~$\phip$ satisfies $\|\phip'\|_{C^{s}((0,1+\de_p))}\leq C\bla$ and is of class $C^s$ in $p$ and $\bla$.
\end{enumerate}
\end{lemma}

\begin{remark}\label{R.local}
  When the equation is linear (that is, $F(x,u)=f(x)$), one can take
  $\bla_p$ arbitrarily large and
\[
\phip(r) = -\frac\bla{2n} f(p) \, (r^2-1)\,.
\]
The operator $\Hp$ is then
\[
\Hp \psi= \frac\bla{n} f(p)\, (\Laz\psi-\psi)\,,
\]
where $\Laz:=[(\frac n2 -1)^2-\De_{\SS }]^{1/2}-\frac n2+1$ is
the Dirichlet--Neumann map of the ball.
\end{remark}

If what follows we shall always assume that $\bla<\bla_p$.

\begin{proposition}\label{P.existence}
For any small enough~$\ep $ and any function~$B\in C^{s+1}(\SS) $ with
$\|B\|_{C^{s+1}}<1$, there is a unique function~$u=\ueB$ in a
small neighborhood of~$f_0(p)+\phip$ in $C^{s+1}(\Ome)$ that satisfies
Equation~\eqref{eqz} and the Dirichlet boundary
condition~\eqref{DBC}.
\end{proposition}

\begin{proof}
Let
$\chip:\BB\to \Ome$ be the diffeomorphism defined in spherical coordinates as
\[
(\rho,\om)\mapsto \big([1+\ep \chi(\rho)\, B(\om)]\rho,\om\big)\,,
\]
where $\chi(\rho)$ is a smooth cutoff function that is zero for $\rho<1/4$ and $1$
for $\rho>1/2$ . Then one can define a map
\[
\cHp: (-\ep_p,\ep_p)\times C^{s+1}\Dir(\BB)\to C^{s-1}(\BB)
\]
as
\begin{equation*}
\cHp(\ep,\phi):= \big[ \tL(\phi\circ \chip^{-1})\big]\circ\chip+E\circ\chip+ \bla
\big[ \tF(\cdot,\widetilde f_0+\phi\circ\chip^{-1})\big]\circ \chip\,,
\end{equation*}
with the function $E$ defined as
\begin{equation}\label{eq_rho}
E:=\tL \widetilde f_0\,.
\end{equation}
Note that $\|E\|_{C^{s-1}(\Omega_{p,\epsilon B})}\leq C\ep^2$ because $\widetilde
f_0(z):=f_0(p+\ep z)$. Clearly, $\cHp(\ep,\phi)=0$ if and only if~$u:=\widetilde f_0+\phi\circ\chip^{-1}$ solves the
Dirichlet problem~\eqref{eqz}-\eqref{DBC} in~$\Ome$.

Note that, by definition and using~\eqref{eq_rho}, $\cHp(0,\phip)=0$. Also, a short computation
shows that the derivative of~$\cHp(\ep,\phi)$ with respect to~$\phi$ satisfies
\[
D_\phi\cHp(0,\phip)= \Tp\,,
\]
so it is an invertible map $C^{s+1}\Dir(\BB)\to C^{s-1}(\BB)$, cf.~Lemma~\ref{L.local}. The
implicit function theorem in Banach spaces then ensures that, for any $\ep$ close
enough to~0, there is a unique function $\phi^\ep$ in a small
neighborhood of~$\phip$ in $C^{s+1}_\mathrm{Dir}(\BB)$ satisfying
\[
\cHp(\ep,\phi^\ep)=0\,.
\]
Then $\ueB:= \widetilde f_0+\phi^\ep\circ\chip^{-1}$ is the desired solution to the
Dirichlet problem in~$\Ome$.
\end{proof}

We will henceforth denote by
\[
  \PP: C^{s+1}(\SS)\to C^{s+1}(\Ome)\]
the map $\psi\mapsto v_\psi$, where $v_\psi$ is the only solution to
the problem
\[
\Tp v_\psi=0\qquad \text{in }\Ome
\]
with the boundary condition
\[
v_\psi(1+\ep B(\om),\om) = \psi(\om)\,.
\]
Note that the existence and uniqueness of~$v_\psi$ is an easy
consequence of Lemma~\ref{L.local}.

For future
reference, let us record here the definition of the associated
Dirichlet--Neumann operator $\Lae: C^{s+1}(\SS)\to C^{s}(\SS)$,
\[
\Lae\psi(\om):= \nu\cdot A\nabla\PP\psi(1+\ep B(\om),\om)\,.
\]
As $\Lae$ reduces to the standard Dirichlet--Neumann map $\Laz$ when
$\ep=\bla=0$, it is standard that
\begin{align}
  \|\Lae-\Laez\|_{C^{s+1}(\SS)\to C^s(\SS)}&\leq C|\ep|\,,\label{laez}\\
  \|\Lae-\Laz\|_{C^{s+1}(\SS)\to C^s(\SS)}&\leq C(|\ep|+\bla)\label{la0}\,.
\end{align}

\section{Asymptotic expansions}
\label{S.Asympt}

In this section we compute asymptotic formulas for the solution to
the Dirichlet problem in the domain~\eqref{Ome} obtained in Proposition~\ref{P.existence}, valid for
$|\ep|\ll1$. Let us begin with the estimates for the solutions to the Dirichlet problem:

\begin{proposition}\label{P.asympt}
  The function~$\ueB$ is of the form
  \[
\ueB= f_0(p)+\phip(r) + \ep\, \Big\{ \Wp(r)\cdot z +\PP\big[\nabla
f_0(p)\cdot\om -\phip'(1)\, B\big]\Big\} + O(\ep^2)\,,
  \]
  where $\Wp:[0,1+\de_p]\to\RR^n$ is a function with
  $\|\Wp\|_{C^{s+1}}\leq C\bla$.
\end{proposition}

 \begin{remark}\label{R.asympt}
 In the case when $F(x,u)=f(x)$, the formula is slightly more
 explicit:
 \begin{multline*}
   \ueB= f_0(p) -\frac{\bla}{2n}f(p)\, (r^2-1) \\
   +\ep\bigg\{ \bigg[\nabla f_0(p) -\frac{\bla (r^2-1)}{2n+4}\Big(\nabla
   f(p)-\frac{f(p)b(p)}{n}\Big)\bigg] \cdot z +\frac{\bla\, f(p)}{n}\mathbb P_{\ep B} B\bigg\} + O(\ep^2)\,.
 \end{multline*}
 Here we are using the notation $\mathbb P_{\ep B}\equiv \mathbb P_{p,0,\ep B}$, which does not depend on $p$ because $F'=0$.
 \end{remark}

\begin{proof}
  Note that $u_0:= f_0(p)+\phip(r)$ satisfies the equation
  \[
\De u_0 + \bla F(p,u_0)=0\,,\qquad u_0|_{r=1}=f_0(p)\,.
\]
Let us write $u_1:= [\ueB - u_0]/\ep$ and observe that
\[
\tF(z,\ueB)= F(p+\ep z,u_0+\ep u_1)= F(p,u_0)+ \ep \big[ \nabla
F(p,u_0)\cdot z + F'(p,u_0) u_1\big] + O(\ep^2)\,.
\]
As $\tL \ueB+ \bla \tF(z,\ueB)=0$ with the boundary condition
\[
\ueB(1+\ep B(\om),\om)= \tf_0(1+\ep B(\om),\om)= f_0(p)+\ep \nabla
f_0(p)\cdot \om + O(\ep^2)\,,
\]
this ensures that $u_1$ satisfies an equation of the form
\begin{align*}
\Tp u_1+ \bla \nabla
F(p,u_0)\cdot z + b(p)\cdot \frac z r\,
  \phip'(r) + O(\ep)=0
\end{align*}
in $\Ome$ and the boundary condition
\[
u_1(1+\ep B(\om),\om)=  \nabla f_0(p)\cdot\om -\phip'(1)\, B(\om)+ O(\ep)\,.
\]

To analyze~$u_1$, we start by noting that
\[
  u_1^*:= \PP[\nabla
  f_0(p)\cdot\om -\phip'(1)\, B(\om)]
\]
satisfies the equation $\Tp u_1^*=0$ in~$\Ome$ and the
boundary condition
\[
u_1^*(1+\ep B(\om),\om)=  \nabla f_0(p)\cdot\om -\phip'(1)\, B(\om)\,.
\]

It is an easy consequence of Lemma~\ref{L.local} that the
equation
\[
\Tp w+\bla \nabla F(p,u_0(|z|))\cdot z +b(p)\cdot \frac{z}{r}\, u_0'(|z|)=0\quad \text{in } \BB\,,\qquad w|_{\pd\BB}=0
\]
has a unique solution~$w$, which is then of the form $w=  \Wp(|z|)\cdot z$ for some $\RR^n$-valued
function $\Wp$. Specifically, its $j$-th component $W_j(r):=
\Wp(r)\cdot e_j$ satisfies the ODE
\[
W_j''(r)+ \frac{n+1}r W_j'(r) +\bla F'(p,u_0(r)) W_j(r)+ \bla \,
\pd_jF(p,u_0(r)) + b_j(p)\, \frac{u_0'(r)}r=0
\]
with the boundary condition $W_j(1)=0$ and the requirement that $W_j$
must be regular at~0. As $u_0(r)$ is well defined up to $r=1+\de_p$,
so is $W_j(r)$. The function~$\Wp$ is obviously bounded as
\[
\|\Wp\|_{C^{s+1}((0,1+\de_p))}\leq C \bla
\|\pd_jF(p,u_0)\|_{C^{s-1}((0,1+\de_p))} + C\bigg\| \frac{u_0'}r\bigg\|_{C^{s-1}((0,1+\de_p))}\,.
\]
Since $\|u_0'\|_{C^{s}((0,1+\de_p))}\leq C\bla$ by Lemma~\ref{L.local}, we infer that
$\|\Wp\|_{C^{s+1}}=O(\bla)$ as well.

By construction, we immediately obtain that $u_1=u_1^*+w+O(\ep)$, so
the proposition follows. The expression of Remark~\ref{R.asympt}
follows from the same argument taking into the account the formula
for~$\phip$ provided in Remark~\ref{R.local}.
\end{proof}

Next we obtain asymptotic formulas for the normal derivative of~$u$:

\begin{proposition}\label{P.Neumann}
  The normal derivative of the function $\ueB$ satisfies
  \[
\pd_\nu \ueB = \phip'(1) + \ep\Big\{ \Hp B +[ \nabla
f_0(p) + \Vp]\cdot \om \Big\} +O(\ep^2)\,,
\]
where the constant vector $\Vp\in\RR^n$ satisfies $|\Vp|\leq C\bla$.
\end{proposition}

 \begin{remark}\label{R.Neumann}
   When $F(x,u)=f(x)$, one can obtain a more compact formula:
   \begin{align}
\nonumber \pd_\nu \ueB = -\frac\bla n f(p) + &\ep\bigg\{ -\frac\bla n f(p)\,
 \big(B-\Laz B\big) \\
 &+ \nabla f_0(p)\cdot \om -\frac{\bla}{n+2}\bigg(\nabla f(p)-\frac{f(p)b(p)}{n}\bigg)\cdot\omega\bigg\} + O(\ep^2)\,.
   \end{align}
   \end{remark}

\begin{proof}
Since the boundary of~$\Ome$ is the zero set of the function
$r-\ep B(\om)-1$, it is clear that its unit normal vector at the point
$(1+\ep B(\om),\om)$ is
\[
\nu= \frac{\om - \frac\ep{1+\ep B(\om)}\nabla_{\SS} B(\om)}{[1+\frac{\ep^2}{(1+\ep B(\om))^2} |\nabla_\SS B(\om)|^2]^{1/2}}
= \om-\ep \nabla_\SS B(\om)+ O(\ep^2)\,,
\]
where $\nabla_\SS$ denotes covariant differentiation on the unit
sphere.

  Using this formula, it follows from Proposition~\ref{P.asympt} that
  \begin{multline*}
    \pd_\nu \ueB= \nu\cdot\nabla \ueB (1+\ep B(\om),\om)= \phip'(1+\ep B(\om)) \\
     + \ep\Big\{  (r\Wp)'(1)\cdot \om +\nu\cdot\nabla \PP\big[\nabla
      f_0(p)\cdot\om -\phip'(1)\, B\big]\Big\} + O(\ep^2)\,.
  \end{multline*}
Since $\phip(r)$ is $C^{s+1}$-smooth for $r<1+\de_p$, let us now
expand $\phip'$ and use the definition of the operator~$\Lae$ to write
\begin{align*}
  \pd_\nu \ueB& =\phip'(1) + \ep\Big\{ \phip''(1) B-\phip'(1)\, \Lae
                B\\
  &\qquad + \Lae\big(\nabla f_0(p)\cdot \omega\big)+
    \Wp'(1)\cdot\om\Big\} + O(\ep^2)\,.
\end{align*}
Let us now recall that $\Hp  B:= \phip''(1) B-\phip'(1)\, \Laez B$ (cf. Lemma~\ref{L.local}) and
that the usual Dirichlet--Neumann map of the ball satisfies $\Laz(V\cdot
\om)=V\cdot \om$ for all $V\in\RR^n$. Therefore, we can use the
bounds~\eqref{laez}-\eqref{la0} and the estimate $|\Vp|\leq C\bla$ with
\[
\Vp:=\Wp'(1)\,,
\]
proven in Proposition~\ref{P.asympt}, to obtain the formula of the
statement. The expression of Remark~\ref{R.Neumann} follows from the
above argument after taking into account the expression for $\ueB$
given in Remark~\ref{R.asympt}.
\end{proof}

\section{Proof of Theorem~\ref{T.1} when $A(x)=I$}
\label{S.T2}

For any given point $p\in\RR^n$, let us now define a map
\[
  \cFp:(-\ep_p,\ep_p)\times X_{s+1}^1\to C^s(\SS)\,,
\]
with $X_s^1:=\{ b\in C^s(\SS): \|b\|_{C^s}<1\}$, as
\[
\cFp(\ep, B):= \pd_\nu\ueB -\frac{\phip'(1)}{f_1(p)}\widetilde f_1\,.
\]
Roughly speaking, this map measures how far the Dirichlet
solution~$\ueB$ is from satisfying the Neumann condition in the domain~$\Ome$ with a
constant
$$\bc:=-\frac{\phip'(1)}{f_1(p)}>0\,.$$

An immediate consequence of the asymptotic formulas for $\pd_\nu \ueB$
proved in Proposition~\ref{P.Neumann} and the fact that
\[
\widetilde f_1(1+\ep B(\om),\om)= f_1(p)+ \ep \nabla f_1(p)\cdot\om +O(\ep^2)\,,
\]
is the following:

\begin{proposition}\label{P.cFp}
  For any fixed~$p\in\RR^n$, any $B\in X_{s+1}^1(\SS)$ and any $|\ep|<\ep_p$,
  \[
    \cFp(\ep, B) = \ep\Bigg\{ \Hp B+\bigg[ \nabla f_0(p)-
    \frac{\phip'(1)}{f_1(p)} \nabla f_1(p) +\Vp\bigg]\cdot\om\Bigg\} + O(\ep^2)\,.
\]
\end{proposition}

 \begin{remark}\label{R.cFp}
   When $F(x,u)=f(x)$, one can obtain a slightly more explicit formula:
   \begin{align}
 \nonumber \cFp(\ep, B) =\ep\Bigg\{ -\frac\bla n f(p)\,
 \big(B-\Laz B\big) +& \bigg[\nabla f_0(p)+
     \frac{\bla\, f(p)}{n\,f_1(p)} \nabla f_1(p)\bigg] \cdot \om\\& -\frac{\bla}{n+2}\bigg[\nabla f(p)-\frac{f(p)b(p)}{n}\bigg]\cdot\omega\Bigg\} + O(\ep^2)\,.
   \end{align}
  \end{remark}

It then follows that the function $\cFp(\ep,B)/\ep$
can be defined at $\ep=0$ by continuity. Furthermore, its derivative
with respect to~$B$ involves the operator~$\Hp$, whose kernel was shown to be the space~$\cK$ in
Lemma~\ref{L.local}. Consequently, let us define the spaces
\[
\cX_s:=\{ b\in C^s(\SS): \cP_{\mathcal K} b=0\}\,,\qquad \cX_s^1:=\{ b\in\cX_s: \|b\|_{C^s}<1\}\,,
\]
with $\cP_\cK$ being the orthogonal projector onto the
subspace~$\cK$. We also define the operator
\[
  \cP b:= b-\cP_{\mathcal K} b\,.
\]
It is clear from these expressions that~$\cP$ maps each
space $C^s(\SS)$ into itself and $\cX_s^1\subset X^1_s$.

By Proposition~\ref{P.cFp}, we can now define a map
\[
\cGp: (-\ep_p,\ep_p)\times \cX_{s+1}^1\to\cX_s
\]
as
\[
\cGp(\ep, B):= \frac{\cP\cFp(\ep, B)}\ep\,.
\]

\begin{lemma}\label{L.vf}
  Let~$U\subset\RR^n$ be any bounded domain. For any $\bla\in(0,\bla_U)$, with 
  $$\bla_U:=\inf_{p\in U}\,\bla_p>0\,,$$ 
  there exist some $\ep_{U,\bla}>0$ and a
  $C^s$~function $Y_{\ep,\bla}: U\to \RR^n$ such
  that
  \[
\pd_\nu \ueBe- \frac{\phip'(1)}{f_1(p)}\widetilde f_1 = \Ye(p)\cdot \om
\]
for all~$p\in U$ and all $|\ep|<\ep_{U,\bla}$. Here $\Ye(p):= Y(\ep,p,\bla)$
is of class $C^s$ in all its arguments, and can be interpreted as a family of parametrized vector fields on~$U$, and
$\Be$ is a certain function in $\cX^1_{s+1}$.
\end{lemma}

\begin{proof}
  Let us begin by showing that the Fr\'echet derivative
  $D_B\cGp(0,0): \cX_{s+1}\to \cX_s$ is one-to-one. To see this, note
  that Proposition~\ref{P.cFp} and the fact that $\cP(A\cdot \om)=0$
  for any $A\in\RR^n$ imply that the derivative of~$\cGp$ with respect
  to~$B$ is of the form
\[
D_B\cGp(\ep,0) = \Hp+ \cE
\]
with $\|\cE\|_{\cX_{s+1}\to\cX_s}\leq C|\ep|$. Here we have used that,
by Lemma~\ref{L.local}, $\cP\Hp=\Hp$ because the range of the elliptic
first-order operator~$\Hp$ is
contained in $\cK^\perp$. The estimate~\eqref{estHp} then ensures that
$D_B\cGp(\ep,0)$ is an invertible map $\cX_{s+1}\to \cX_s$ provided that
$\ep$ is small enough.

As $\cGp(0,0)=0$, the invertibility of $D_B\cGp(\ep,0)$ implies, via the
implicit function theorem, that for any small enough~$\ep$, there is a
unique function~$\Be$ in a small neighborhood of~$0$ such that
\[
\cGp(\ep,\Be)=0\,.
\]
By the definition of~$\cFp$ and the fact that $\cK=\{Y\cdot\om:
Y\in\RR^n\}$, this implies that there is some $Y(\ep,p,\bla)\in\RR^n$
such that
\[
\pd_\nu \ueBe - \frac{\phip'(1)}{f_1(p)}\widetilde f_1= Y(\ep,p,\bla)\cdot\om\,.
\]
Furthermore, $Y(\ep,p,\bla)$ is a $C^s$-smooth function of its
arguments because so is the left hand side of this identity.
\end{proof}

Let us now note that the asymptotic expression of the vector
field~$\Ye(p)$ can be read off Proposition~\ref{P.cFp}:

\begin{lemma}\label{L.asymptY}
  The vector field~$\Ye$ is of the form
  \[
\Ye(p)= \ep \bigg[ \nabla f_0(p)-
    \frac{\phip'(1)}{f_1(p)} \nabla f_1(p) +\Vp\bigg]+O(\ep^2)\,.
  \]
When $F(x,u)=f(x)$, one can write down the more precise expression
  \[
 \Ye(p)= \ep \bigg\{ \nabla f_0(p)+
     \frac{\bla\, f(p)}{n\,f_1(p)} \nabla f_1(p) -\frac{\bla}{n+2}\bigg[\nabla f(p)-\frac{f(p)b(p)}{n}\bigg]\bigg\}+O(\ep^2)\,.
     \]
  \end{lemma}

  \begin{proof}[Proof of Theorem~\ref{T.1} when $A(x)=I$ and of Remark~\ref{R.Gka}]
Let us suppose that $p^*$ is a nondegenerate critical point of the
function~$f_0$. As $\phip'(1)=O(\bla)$ by Lemma~\ref{L.local},
Lemma~\ref{L.asymptY} implies that
\[
\frac{\Ye(p)}\ep = \nabla f_0(p) + \cE
\]
with an error bounded as $\|\cE\|_{C^1(U)} \leq C_U|\bla|+ C_U|\ep|$ for any bounded domain $U\ni p^*$.
If $|\bla|$ and~$|\ep|$ are small enough, it is then standard that
there is a unique point $\pe$ in a small neighborhood of~$p^*$ such that
\[
\Ye(\pe)=0\,.
\]
By Lemma~\ref{L.vf}, and setting $\bc:=-\phi_{\pe,\bla}'(1)/f_1(\pe)$, this ensures that
\[
\pd_\nu u_{\ep,\pe,\bla,\Be}+ \bc \widetilde f_1\equiv 0\,,
\]
which implies the claim of the theorem with the domain
$\Om_{\pe,\ep \Be}$.

To prove Remark~\ref{R.Gka} on overdetermined solutions for the torsion problem, let us assume that $F(x,u)=f(x)=1$ and that $p^*$ is a nondegenerate
critical point of the function $f_0+\ka \ln f_1$ for some constant
$\ka>0$. In this case, since $f(x)=1$ and $b(x)=0$, Lemma~\ref{L.asymptY} implies that
\[
\frac{\Ye(p)}\ep = \nabla f_0(p) +\frac{\bla}{n}\,{\nabla \ln f_1(p)} + \cE'
\]
with $\|\cE'\|_{C^1(U)} \leq C_U\ep$. As one
can pick any positive value of $\bla$ by Remark~\ref{R.local}, let us
fix $\bla=\bla^*:= n\ka>0$. The previous
argument then allows us to conclude that, for any small enough~$\ep$,
there exists some point $p_\ep$ close to~$p^*$ for which $\Ye(p_\ep)=0$. As above,
this implies the existence of solutions to the overdetermined
torsion problem. The case of $f_0=0$, $f_1=1$ and $F(x,u)=f(x)$ is
handled similarly, so Remark~\ref{R.Gka} then follows.
\end{proof}

\section{Introduction of a nonconstant matrix $A(x)$ and conclusion of
  the proof}
\label{S.A}

In this section we will show why the proof of Theorem~\ref{T.1}
carried out for the case when $A(x)=I$ remains valid, with only minor
variations, in the case of a general matrix~$A(x)$.

The key idea is that we are constructing domains that are small
deformations of the ball of radius~$\ep$, with $\ep\ll1$. Over scales
of order~$\ep$, the function~$A(x)$ is essentially constant, so it
stands to reason that one might be able to compensate the effect of
having a nonconstant matrix~$A(x)$ (at least, to some orders when
considering an asymptotic expansion in~$\ep$) by deforming the balls
accordingly. More visually, this would correspond morally to picking
an ellipsoidal domain at each point~$x$ with axes determined by the
matrix~$A(x)$.

The way to implement this idea is through (a rescaling of) the normal coordinates
associated with the matrix-valued function~$A$, which we now regard as a Riemannian
metric on~$\RR^n$ of class~$C^{s+2}$. These are defined through the
exponential map at a point~$p\in\RR^n$,
\[
  \exp_p^A: U_p\to\RR^n\,,
\]
which
maps a certain domain~$U_p\subset\RR^n$ diffeomorphically onto its
image. It is standard~\cite{DK} that $\exp_p^A(Z)$ is a~$C^{s+1}$ function
of~$Z\in U_p$ and of~$p\in\RR^n$. The normal coordinates at~$p$ are just the Cartesian
coordinates $Z=(Z_1,\dots, Z_n)$ on~$U_p\subset \RR^n$. In these
coordinates, the metric reads as $\widehat A(Z)=I+ O(|Z|^2)$. More precisely, $\widehat A(Z)=(\widehat a_{ij}(Z))$ is given by the pullback by the
exponential map of the metric tensor, which is well-known to be of the form
\[
(\exp_p^A)^* \big[ a_{ij}(x) \, dx_i \, dx_j\big] =: \ha_{ij}(Z)\, dZ_i\, dZ_j
\]
with functions $\ha_{ij}$ of class $C^{s}(U_p)$ such that
\[
\ha_{ij}(0)=\de_{ij}\,,\qquad \pd_{Z_k} \ha_{ij}(0)=0\,.
\]
Therefore, normal coordinates enable us to write the matrix as the
identity plus a $C^{s}$-smooth quadratic error. Incidentally, it is well known
that the leading order
contribution of the error is determined by the curvature of the
metric~$A$ at the point~$p$.

We are now ready to reformulate the overdetermined problem with a
general function~$A$ as a small perturbation of the case $A(x)=I$. For each function~$B\in C^{s+1}(\SS)$ with $\|B\|_{C^{s+1}}<1$
and each small enough~$\ep$, one can then define the domain
$\Ome\subset\RR^n$ (which will play the same role as~\eqref{Ome}) as
\[
\Ome:=\big\{\exp_p^A(\ep z): |z|< 1+\ep B(z/|z|)\big\}\,.
\]
Note that, in terms of the spherical coordinates associated with a point $z$,
\[
r:=|z|\in (0,\infty)\,,\qquad \om:= \frac z{|z|}\in\SS\,,
\]
the above condition reads simply as $r< 1+\ep B(\om)$. In the
domain~$\Ome$, Equation~\eqref{eq} reads in the rescaled
normal coordinates~$z$ at~$p$ as
\[
\hL u + \bla \hF(z,u)=0\,,
\]
where $\hF(z,u):= F(\exp_p^A(\ep z),u)$ and now the linear operator~$\hL$ is of the form
\[
\hL u := \ha_{ij}(\ep z)\, \pd_{z_i z_j}u + \ep \hb_i(\ep z)\,\pd_{z_i} u
\]
with $\ha_{ij}(Z)$ as above and some functions $\hb_i(Z)$ of class~$C^s$.

Therefore,
\[
\hL u = \De u  + \ep \hb_i(\ep z)\,\pd_{z_i} u+\cE u \,,
\]
where the error term is bounded as $\|\cE u\|_{C^{s-1}}\leq
C\ep^2\|u\|_{C^{s+1}}$ and $\hL u-\cE u$ is just like the
operator~$\tL u$ introduced below Equation~\eqref{eqz}. One can now go over the
proof of Theorem~\ref{T.1} and readily see that all the arguments
remain valid when one introduces an error of this form in the
expressions. This is not surprising, as the proof only uses the
formulas for the terms in the equations that are of zeroth and
first order in~$\ep$. Since the nondegenerate critical points of $f_0$ do not depend on the coordinate system, Theorem~\ref{T.1} is then proven for a general matrix-valued function~$A$.

\section*{Acknowledgements}

M.D.-V.\ is supported by the grants MTM2016-75897-P (AEI)
and ED431F 2017/03, and by the Ram\'on y Cajal program of the Spanish Ministry
of Science. A.E.\ is supported by the ERC Starting
Grant~633152. D.P.-S.\
is supported by the grants MTM PID2019-106715GB-C21 (MICINN) and Europa Excelencia EUR2019-103821 (MCIU). This work is supported in part by the
ICMAT--Severo Ochoa grant
SEV-2015-0554 and the CSIC grant 20205CEX001.

\bibliographystyle{amsplain}

\end{document}